\documentclass{amsart}
\pdfoutput=1
\usepackage{amssymb, amsmath, amscd, latexsym, mathrsfs, appendix, mathdots}
\usepackage{graphics}
\usepackage{amsmath}
\usepackage{amsxtra}
\usepackage{amsfonts}
\usepackage{appendix}
\usepackage[all]{xy}

\usepackage{color}

\usepackage{hyperref}
\hypersetup{
    colorlinks,
    citecolor=black,
    filecolor=black,
    linkcolor=black,
    urlcolor=black
}

\numberwithin{equation}{section}
\newtheorem{thm}{Theorem}[section]
\newtheorem{cor}[thm]{Corollary}
\newtheorem{lem}[thm]{Lemma}
\newtheorem{prop}[thm]{Proposition}

\theoremstyle{definition}
\newtheorem{defn}[thm]{Definition}

\newtheorem{rem}[thm]{Remark}
\newtheorem{expl}[thm]{Example}

\newtheorem{notn}[thm]{Notation}

\newcommand{\lra}{\longrightarrow}

\newcommand{\co}{\colon\!}

\newcommand{\smin}{\smallsetminus}

\newcommand{\id}{\textup{id}}
\newcommand{\im}{\textup{im}}

\newcommand{\map}{\textup{map}}

\newcommand{\config}{\mathsf{con}} 

\newcommand{\cconfig}{\mathsf{ucon}}

\newcommand{\fin}{\mathsf{Fin}}
\newcommand{\cfin}{\mathsf{uFin}}

\newcommand{\skel}{\textup{sk}}

\newcommand{\sA}{\mathcal A}
\newcommand{\sB}{\mathcal B}

\newcommand{\sD}{\mathcal D}

\newcommand{\op}{\textup{op}}

\newcommand{\LL}{\mathbb L}
\newcommand{\FF}{\mathbb F}

\newcommand{\NN}{\mathbb N}
\newcommand{\RR}{\mathbb R}
\newcommand{\ZZ}{\mathbb Z}

\newcommand{\uli}{\underline}

\makeatletter
\DeclareRobustCommand{\rvdots}{%
 \vbox{
   \baselineskip4\p@\lineskiplimit\z@
   \kern-\p@
   \hbox{.}\hbox{.}\hbox{.}
     }}
\makeatother

\begin{document}

\title{Presentations of configuration categories}
\author{Pedro Boavida de Brito and Michael S. Weiss}%

\address{Dept. of Mathematics, Instituto Superior Tecnico, Univ.~of Lisbon, Av.~Rovisco Pais, Lisboa, Portugal}%
\email{pedrobbrito@tecnico.ulisboa.pt}

\address{Math. Institut, Universit\"at M\"{u}nster, 48149 M\"{u}nster, Einsteinstrasse 62, Germany}%
 \email{m.weiss@uni-muenster.de}

\thanks{The project was funded by the Deutsche Forschungsgemeinschaft (DFG, German Research Foundation) – Project-ID 427320536 – SFB 1442,
as well as under Germany’s Excellence Strategy EXC 2044 390685587, Mathematics Münster: Dynamics–Geometry–Structure.
P.B. was supported by FCT 2021.01497.CEECIND and grant SFRH/BPD/99841/2014.}

\subjclass[2000]{57R40, 55U40, 55P48}
\begin{abstract} The configuration category of a manifold is a topological category
which we view as a Segal space, via the nerve construction.
Our main result is that the unordered configuration category, suitably truncated,
admits a finite presentation as a complete Segal space
if the manifold in question is the interior of a compact manifold.
\end{abstract}
\maketitle

\section{Finite presentation properties of complete Segal spaces}
\begin{defn} A \emph{presentation} of a complete Segal space $Y$ is a map
of simplicial spaces $X\to Y$ which is a weak equivalence in the model structure of
complete Segal spaces \cite{Rezk01}. If $X$ is homotopically compact as a simplicial space, we speak of a \emph{finite presentation}.
\end{defn}

By \emph{homotopically compact as a simplicial space} we mean the following. Let $\Delta_{\leq k} \subset \Delta$
be the full subcategory spanned by the objects $[n]$ where $n \leq k$. We
say that $X$ is homotopically compact as a simplicial space if $X_n$ is homotopy equivalent
to a compact CW-space for each $n$, and there is some $k \geq 0$ such that, for every simplicial space $Y$, the induced map
\[
\RR \map(X,Y) \to \RR \map(X|_{\Delta_{\leq k}},Y|_{\Delta_{\leq k}})
\]
is a weak equivalence. The derived mapping spaces are formed in a setting
where a map $X\to Y$ of simplicial spaces counts as a weak equivalence if and only if the maps $X_n\to Y_n$
are weak homotopy equivalences for all $n\ge 0$, respectively, all $n\in\{0,1,\dots,k\}$.

\begin{expl} \label{expl-insur} Suppose that a complete Segal space $Z$ is homotopically compact
as a simplicial space. Then clearly $\id\co Z\to Z$ qualifies as a finite presentation.

Here are three cases of this which are of some interest to us. Let $M$ be the interior of a
compact smooth manifold and let $\alpha$ be a positive integer. Let $\cconfig(M;\alpha)$ be the
unordered and truncated configuration category of $M$, so that only configurations of cardinality $\le \alpha$
are allowed. This is a complete Segal space. It is the unordered variant of $\config(M;\alpha)$ in \cite{BoavidaWeissLong}.
(There are various models available; we authors tend to prefer
the particle model, in which $\cconfig(M;\alpha)$ is the nerve of a topological category, a.k.a. category
object in the category of topological spaces. See \S2.)

The configuration category has a complete Segal subspace ${\cconfig}^{-}\!(M;\alpha)$ obtained 
by allowing all objects, but only those
morphisms whose underlying map of finite sets (finite subsets of $M$) is surjective. It is easy to see that ${\cconfig}^{-}\!(M;\alpha)$ is
homotopically compact as a simplicial space. (As a Segal space in its own right, ${\cconfig}^{-}\!(M;\alpha)$ is popular in factorization homology.)
There is another complete Segal subspace $\cconfig^+\!(M;\alpha)$ of $\cconfig(M;\alpha)$ obtained by allowing
only those morphisms whose underlying map of finite sets is injective. Again, $\cconfig^+\!(M;\alpha)$ is
homotopically compact as a simplicial space. And finally $\cconfig^-\!(M;\alpha)\cap \cconfig^+\!(M;\alpha)$ is homotopically compact as a simplicial space; it is also
essentially constant as a simplicial space. By contrast, $\cconfig(M;\alpha)$ itself is usually not homotopically compact
\emph{as a simplicial space}. But we can still ask whether it is finitely presentable as a Segal space.
\end{expl}

\begin{lem} Let $Y$ be a complete Segal space which is finitely presentable. Let $f\co W\to Y$ be a map
of simplicial spaces, where $W$ is homotopically compact as a simplicial space. Then $f$ admits a
factorization
\[
\xymatrix{
W \ar[r]^-{f_1} & W' \ar[r]^-{f_2} & Y
}
\]
in which $f_1$ is a cofibration, $W'$ is a homotopically compact simplicial space and $f_2$ qualifies as a
presentation of $Y$.
\end{lem}

\proof  (Sketch.) Start with a finite presentation $g\co X\to Y$.
As in \cite[I, Thm 4.3.8]{Hirschhorn}, there is a sequence of cofibrations
\[
X=X(0)\hookrightarrow X(1) \hookrightarrow X(2) \hookrightarrow \cdots
\]
where each $X(j)$ is homotopically compact, and a (degreewise) weak equivalence of simplicial
spaces $\bar g\co \bigcup_n X(n) \to Y$ which extends $g$, in such a way
that $X(j-1)\hookrightarrow X(j)$ is a weak equivalence in the model structure
of complete Segal spaces, for all $j\ge 1$.
Therefore $W'$ can be taken to be $X(r)$ (essentially) for some large $r$.
\qed


\begin{cor} A homotopy pushout of finitely presentable complete Segal spaces (in the model category of such)
is again finitely presentable. \qed
\end{cor}

\section{Generators and relations for configuration categories}\label{sec-genrel}

Let $\fin$ denote the skeleton of the category of finite sets with objects
$\uli n = \{0,1, \dots, n\}$, for $n \geq 0$. Since most objects of $\fin$ have non-trivial automorphisms, the nerve
$N\fin$ is not a \emph{complete} Segal space. We denote its completion by $\cfin$; the letter $u$ is for \emph{unordered}.
Therefore $\cfin$ in degree $n$ is the classifying space of the groupoid of
functors $[n] \to \fin$ and their isomorphisms. (We will see some alternative models later, for example in
remark~\ref{rem-otherfincp}.)
In particular, the space of $0$-simplices of
$\cfin$ is the disjoint union, over $m$, of $B \Sigma_m$. We let $\FF$ denote the simplicial (discrete)
space obtained by applying $\pi_0$ degreewise to $\cfin$. Given a morphism $f$ in $\fin$, we write $[f]$ for
its image in $\FF$. (Beware: $\FF$ is not the nerve of any category, in other words, it is not a discrete Segal space.
We use it for bookkeeping.)

We are now ready to define $\cconfig(M)$, the unordered configuration category of a topological manifold $M$.
More precisely, this is the particle model of $\cconfig(M)$.

\begin{defn} An object in $\cconfig(M)$ is a finite subset $S \subset M$. A morphism from $S \subset M$ to $T \subset M$
consists of a map $g\co S\to T$ and a (Moore) homotopy $h\co S\times[0,a]\to M$ such that $h_0$ is the inclusion, $h_a$ is
$g$ followed by $T\hookrightarrow M$, and $h$ satisfies the inertia condition: if $h_s(x)=h_s(y)$ for some $x,y\in S$ and
$s\in [0,a]$, then $h_t(x)=h_t(y)$ for all $t\in [s,a]$. \end{defn}

\begin{prop}\label{prop-ucompletion}
The quotient map $\config(M) \to \cconfig(M)$ is a Rezk completion.
\end{prop}
\begin{proof}
It is clear that $\cconfig(M)$ is complete. To verify that the map is a Rezk completion, we can show that it is a
Dwyer-Kan equivalence. The map is essentially surjective, because it is surjective on objects.

The quotient map on objects $\config(M)_0 \to \cconfig(M)_0$ is a fibration (in fact, a disjoint union of
covering spaces, one for each $n \geq 0$ with fiber $\Sigma_n$). The same is true for the map
$\config(M)_1 \to \cconfig(M)_1$ on morphism spaces. Hence, full faithfulness amounts to the statement that the square
\[
\xymatrix{ \config(M)_1 \ar[d]^-{(d_0,d_1)} \ar[r] & \cconfig(M)_1 \ar[d]^-{(d_0,d_1)} \\
\config(M)_0  \times \config(M)_0 \ar[r] & \cconfig(M)_0 \times \cconfig(M)_0
}
\]

is cartesian. This follows easily by inspection of the horizontal fibers.
\end{proof}

\begin{rem} \label{rem-otherfincp}
The standard map of Segal spaces from $\config(M)$ to $\fin$ induces a map of the completed
Segal spaces $\cconfig(M) \to \cfin$. We like to describe this as follows: it is the map
from $\cconfig(M)$ to $\cconfig(\RR^\infty) $ induced by some embedding $M\to \RR^\infty$.
Here $\cconfig(\RR^\infty)$ is the (sequential) homotopy colimit of the Segal spaces $\cconfig(\RR^s)$ for $s\to \infty$.
\end{rem}

By construction, there is a commutative square of Segal spaces
\[
\xymatrix@R=16pt{
\config(M) \ar[d] \ar[r] & \ar[d] \cconfig(M) \\
	\config(\RR^\infty) \ar[r] & \cconfig(\RR^\infty)
}
\]
where the vertical arrows are induced by $M\hookrightarrow \RR^\infty$.
This square is a degreewise pullback square and also a degreewise homotopy pullback square.
(We may also write $\fin\to \cfin$ for the lower row.)
This implies the following.

\begin{prop}
The map
\[
\RR \map_{\fin}(\config(M), \config(N)) \to \RR \map_{\cfin}(\cconfig(M), \cconfig(N))
\]
determined by completion is a weak equivalence. The derived mapping spaces are taken with respect to degreewise weak equivalences
of simplicial spaces.  \qed
\end{prop} We also have the truncated variant $\cconfig(M;\alpha)$ where $\alpha$ is a positive integer. One of our main goals is to
prove the following. (The proof is in section~\ref{sec-def}.)

\begin{thm} \label{thm-finpres} If $M$ is the interior of a compact topological manifold, then the complete
Segal space $\cconfig(M;\alpha)$ is finitely presentable.
\end{thm}

We also aim for an enlightening description of $\cconfig(M;\alpha)$ and the un-truncated $\cconfig(M)$ in terms of (few) generators and
relations. We begin with a general investigation of $\cconfig(M)$ from this point of view.

\smallskip
The generators and relations that we have in mind are indexed by isomorphism classes of commutative diagrams
of finite sets and maps
of the shape
\begin{equation} \label{eqn-filled}
\begin{aligned}
\xymatrix@M=7pt@R=14pt@C=18pt{
S_{0,s} \ar@{<-<}[r] & S_{1,s}  \ar@{<-<}[r]  & \cdots  \ar@{<-<}[r] & S_{r-1,s} \ar@{<-<}[r] & S_{r,s} \\
S_{0,s-1} \ar@{<<-}[u] \ar@{<-<}[r] & S_{1,s-1} \ar@{<<-}[u]  \ar@{<-<}[r] & \cdots  \ar@{<-<}[r] & S_{r-1,s-1} \ar@{<<-}[u] \ar@{<-<}[r] & S_{r,s-1} \ar@{<<-}[u] \\
S_{0,s-2} \ar@{<<-}[u] \ar@{<-<}[r] & S_{1,s-2} \ar@{<<-}[u]  \ar@{<-<}[r] & \cdots  \ar@{<-<}[r] & S_{r-1,s-2} \ar@{<<-}[u] \ar@{<-<}[r] & S_{r,s-2} \ar@{<<-}[u] \\
 \rvdots \ar@{<<-}[u]  & \rvdots  \ar@{<<-}[u]   & \cdots  & \rvdots \ar@{<<-}[u] & \rvdots \ar@{<<-}[u] \\
S_{0,1} \ar@{<<-}[u] \ar@{<-<}[r] & S_{1,1}  \ar@{<<-}[u] \ar@{<-<}[r]  & \cdots  \ar@{<-<}[r] & S_{r-1,1} \ar@{<<-}[u] \ar@{<-<}[r] & S_{r,1} \ar@{<<-}[u] \\
S_{0,0} \ar@{<<-}[u] \ar@{<-<}[r] & S_{1,0} \ar@{<<-}[u] \ar@{<-<}[r] & \cdots  \ar@{<-<}[r] & S_{r-1,0} \ar@{<<-}[u] \ar@{<-<}[r] & S_{r,0} \ar@{<<-}[u]
}
\end{aligned}
\end{equation}
in which all horizontal arrows are injective and all vertical arrows are surjective. Such a diagram,
call it $\sD$, is determined (up to unique isomorphism) by the sub-diagram
\begin{equation} \label{eqn-edges}
\begin{aligned}
\xymatrix@M=7pt@R=14pt@C=18pt{
S_{0,s} \ar@{<-<}[r] & S_{1,s} \ar@{<-<}[r]  & \cdots  \ar@{<-<}[r] & S_{r-1,s} \ar@{<-<}[r] & S_{r,s} \\
 S_{0,s-1} \ar@{<<-}[u] & & & & \\
 \rvdots  \ar@{<<-}[u]  & & & & \\
S_{0,1}  \ar@{<<-}[u] & & & & \\
S_{0,0}  \ar@{<<-}[u] &  & & &
}
\end{aligned}
\end{equation}
which we denote $\sD^\natural$. Often  we will assume that none of the arrows in $\sD^\natural$ are bijective; this amounts
to a nondegeneracy assumption. If we suppose, in addition, that all the finite sets which appear in diagram $\sD^\natural$
have cardinality $\le \alpha$, where $\alpha$ is a fixed positive integer, then the number of isomorphism classes of such diagrams
is clearly finite.

\begin{notn} \label{not:diagrams} $\rule{0mm}{3mm}$  
\begin{itemize}
\item[(i)] Write $\Delta[r]$ for the representable simplicial set represented by the object $[r]$ of $\Delta$.
It is permitted to regard this as the nerve of $[r]^\op$. We tend to view $\Delta[r]$ as a simplicial (discrete)
space. Write $\partial(\Delta[r])$ for the boundary, the simplicial subset generated by the proper faces
of the unique nondegenerate simplex in degree $r$. We may also write
$\partial(\Delta[r]\times\Delta[s])$ for the boundary of $\Delta[r]\times\Delta[s]$, i.e.,
the union of $\partial(\Delta[r])\times\Delta[s]$ and $\Delta[r]\times\partial(\Delta[s])$.
\item[(ii)] For a simplicial space $X$ and a space $K$, let $X\otimes K$ be the simplicial space obtained from $X$ by
taking product with $K$ in every degree.
\item[(iii)] Let $R(\sD)$ be the space of derived maps from $\Delta[r]\times\Delta[s]$ to $\cconfig(M)$
which realize diagram $\sD$, regarded as a map of simplicial sets from $\Delta[r] \times \Delta[s]$ to $\FF$. (By the remark about
$\sD^\natural$, the space $R(\sD)$ is weakly equivalent to a union of connected
components of $\cconfig(M)_{r+s}$; see also
proposition \ref{prop:epimono}.) Let $R(r,s)$ be the disjoint
union of $R(\sD)$ for such $\sD$.
\item[(iv)] For a diagram $\sD$ as above, we may
use the (non-calligraphic) letter $D$ to denote the image of the associated
map $\Delta[r]\times\Delta[s]\to \FF$, so that $D$ is a simplicial subset of $\FF$.
\item[(v)] Let $C\subset \FF$ be a simplicial subset. We write $\cconfig(M|C)$ for the preimage of $C$
under the reference map $\cconfig(M) \to \FF$.
\end{itemize}
\end{notn}

\begin{defn} A simplicial subset $C$ of $\FF$ is \emph{accessible} if it is a union of simplicial
subsets of the form $D=\im(\sD)$, notation~\ref{not:diagrams} (iv); $\sD$ is a diagram as in~\eqref{eqn-filled}.
In such a case, let $R^C(r,s)$ be the disjoint union of the $R(\sD)$, notation~\ref{not:diagrams} (iii), such that $D\subset C$.
\end{defn}
It turns out that $\FF$ itself is accessible (corollary~\ref{cor-defdiag}).

\begin{thm} \label{thm-hocoend}  For an accessible simplicial subset $C$ of $\FF$, the
tautological map
\[  R^C(r,s) \otimes^{\LL}_{([r],[s]) \in \Delta \times \Delta} (\Delta[r]\times\Delta[s]) \lra \cconfig(M|C)  \]
is a weak equivalence in the model structure of Segal spaces.
\end{thm}
This will be proved in section~\ref{sec-epimono}. The $\otimes^{\LL}$ notation
is for a homotopy coend, a.k.a.~two-sided bar construction,
which we make from the covariant functor $([r],[s])\mapsto \Delta[r]\times\Delta[s]$
from $\Delta\times\Delta$ to the category of simplicial spaces and the contravariant functor
$([r],[s])\mapsto R^C(r,s)$.

\begin{cor} \label{cor-hocoend} For an accessible simplicial subset $C$ of $\FF$, the
tautological map
\[  R^C(r,r) \otimes^{\LL}_{[r] \in \Delta} (\Delta[r]\times\Delta[r]) \lra \cconfig(M|C)  \]
is a weak equivalence in the model structure of Segal spaces.
\end{cor}
\begin{proof}
It is well known that the diagonal functor $\Delta^{\op} \to \Delta^{\op}  \times \Delta^{\op}$ is (homotopy) terminal.
\end{proof}

Theorem~\ref{thm-hocoend} can be applied with $C$ equal to $\FF$.
In that case it sheds some light on the relationship between
$\cconfig(M)$ and the Segal subspaces of $\cconfig(M;\alpha)$ which we called $\cconfig^{-}\!(M;\alpha)$ and
$\cconfig^{+}\!(M;\alpha)$ in example~\ref{expl-insur}. Write $\cconfig^{-}\!(M)$ for the union of the
$\cconfig^{-}\!(M;\alpha)$ where $\alpha\ge 1$, and $\cconfig^{+}\!(M)$ for the union of the $\cconfig^{+}\!(M;\alpha)$.
Then $\cconfig^{-}\!(M)$
corresponds, in the expression which the theorem has for $\cconfig(M)$, to the simplicial subspace
\[  R(0,s)   \otimes^{\LL}_{[s] \in \Delta} (\Delta[0]\times\Delta[s]) \]
and $\cconfig^{+}\!(M)$ corresponds to the simplicial subspace
\[ R(r,0) \otimes^{\LL}_{[r] \in \Delta} (\Delta[r]\times\Delta[0]). \]
In this example we have no finite generation statement and it would not make sense.
Nevertheless a connection can be made between theorem~\ref{thm-finpres} and
theorem~\ref{thm-hocoend} or corollary~\ref{cor-hocoend}. For that, fix $\alpha\ge 0$ and let $E^\alpha\subset \FF$ be the simplicial
subset which is the union of all $D$ for diagrams $\sD$ as in~\eqref{eqn-filled} where the cardinality
of all sets involved is $\le\alpha$. By construction, $E^\alpha$ is accessible. In section~\ref{sec-def} we show that the inclusion
$\cconfig(M|E^\alpha) \to \cconfig(M;\alpha)$ is a weak equivalence in the
model category structure of Segal spaces. Moreover
it is clear that
\[  R^{E^\alpha}(r,r) \otimes^{\LL}_{[r] \in \Delta} (\Delta[r]\times\Delta[r]) \]
is homotopically compact as a simplicial space if $M$ is the interior of a compact mani\-fold, essentially
because $E^\alpha$ is finitely generated as a simplicial set. In this way
corollary~\ref{cor-hocoend} specializes to a ``preferred'' finite presentation of $\cconfig(M;\alpha)$.


\section{Epi-mono factorizations} \label{sec-epimono}

In this section we prove theorem \ref{thm-hocoend}. The main step is proposition \ref{prop-addrel}, below. We make a
start by isolating a simple but important observation.

\begin{prop}\label{prop:epimono}
Let $S_0 \stackrel{f_1}{\leftarrowtail} S_1 \stackrel{f_2}{\twoheadleftarrow} S_2$ be a diagram of finite sets,
$f_2$ surjective and $f_1$ injective. Write $f=f_1f_2$. Let $\cconfig(M)_{(f_1,f_2)}$ denote the part of $\cconfig(M)_2$ over the
$2$-simplex of $\FF$ defined by this diagram and let
$\cconfig(M)_f$ denote the part of $\cconfig(M)_1$ over the 1-simplex determined by $f$. Then the composition operator
\[
d_1 : \cconfig(M)_{(f_1,f_2)} \to \cconfig(M)_f
\]
is a weak equivalence.
\end{prop}

\begin{proof} There is an analogous statement for $\config(M)$, and the two are easily shown to be equivalent.
In the $\config(M)$ statement, the diagram has the form
\[  \uli m \stackrel{f_1}{\leftarrowtail} \uli\ell \stackrel{f_2}{\twoheadleftarrow} \uli k \]
The map $d_1\co\config(M)_{(f_1,f_2)} \to \config(M)_f$
is a map over $\config(M)_0$ by dint of the ultimate target map. It follows from the
description of the homotopy fibers of the ultimate target map given in \cite{BoavidaWeissLong} that the induced map on homotopy
fibers is a weak equivalence.
\end{proof}

\subsection{Attaching generators and relations}
For $t>0$ and a $t$-simplex $z$ of $\FF$ represented by a diagram of finite sets and maps
\[
\xymatrix@C=20pt@M=5pt{
S_0 & \ar[l]_-{f_1}  S_1 & \ar[l]_-{f_2}  S_2 & \ar[l]  \cdots & \ar[l]_-{f_{t-1}}  S_{t-1} & \ar[l]_-{f_t}  S_t
}
\]
let $e(z)$ be the $(2t)$-simplex in $\FF$ obtained from that same diagram by decomposing each map $f_i$ into two maps, surjection
followed by injection.
\begin{defn} A simplicial subset $C$ of $\FF$ is \emph{saturated} if for every $t>0$ and every
$t$-simplex $z$ of $C$, the $(2t)$-simplex $e(z)$ of $\FF$ also belongs to $C$.
\end{defn}
\begin{lem} For simplicial subsets of $\FF$, accessible implies saturated. \qed
\end{lem}
A union of saturated simplicial subsets of $\FF$ is always saturated, and a union of accessible
simplicial subsets of $\FF$ is always accessible.

\smallskip
In the next proposition we use notation~\ref{not:diagrams} (ii), but in fact we extend or abuse it slightly and
write $\otimes^\LL$ instead of $\otimes$ to indicate that cofibrant replacement must be inflicted on the
simplicial space factor first. (This refers to the ``projective'' model category structure on the category of simplicial
spaces where the weak
equivalences and fibrations are defined levelwise.)

\begin{prop} \label{prop-addrel} Let $C$ be a saturated simplicial subset of $\FF$.
Suppose that $y$ is an $(r+s)$-simplex of $\FF$ represented by a diagram
$\sD^\natural$ as in \emph{(\ref{eqn-edges})}. Suppose that all proper faces of $y$ belong to
$C$, but $y$ itself does not belong to $C$.
Then there is a commutative square
\[
\xymatrix{
\partial(\Delta[r] \times \Delta[s]) \otimes^\LL R(\mathcal{D}) \ar[r] \ar[d] & (\Delta[r] \times \Delta[s]) \otimes^\LL R(\mathcal{D}) \ar[d] \\
	\cconfig(M|C) \ar[r] & \cconfig(M|C\cup D)
}
\]
with tautological vertical maps, which is a homotopy pushout square of simplicial spaces in the Segal space
model structure.
\end{prop}

In order to prove the proposition, we will require some preparatory results. First an easy one that which will be
used repeatedly.

\begin{lem}\label{lem-coexcisive}
The assignment $C \mapsto \cconfig(M|C)$ is co-excisive, i.e., it sends unions of simplicial subsets of $\FF$ to
(degreewise) homotopy pushouts of simplicial spaces. \qed
\end{lem}

\subsection{Products of simplices}
The simplicial set $\Delta[r] \times \Delta[s]$ has a well-known description
in terms of simplicial subsets isomorphic to $\Delta[r+s]$. Namely,
\[
\Delta[r] \times \Delta[s] \; = \; \bigcup_{\sigma} \Delta[r+s]_{\sigma}
\]
where $\sigma$ runs over the $(r,s)$-shuffles. Recall that a $(r,s)$-shuffle $\sigma$ is an order preserving monomorphism
$[r+s] \hookrightarrow [r] \times [s]$, with the product (partial) ordering on $[r]\times[s]$.
Injectivity implies that $\sigma(0) = (0,0)$, that $\sigma(r+s) = (r,s)$ and
that, for each $k$, the pair $\sigma(k+1)$ is equal to $\sigma(k)+(1,0)$ or to $\sigma(k)+(0,1)$. (Think snake.)

Shuffles form a poset: given two such, $\sigma$ and $\tau$, we say $\sigma \leq \tau$ if $p_2\circ\sigma\leq p_2\circ\tau$,
where $p_2\co [r]\times[s]\to [s]$ is the (second) projection. This poset has a minimal element (denoted $\lrcorner$)
and a maximal element (denoted $\ulcorner$).

For a shuffle $\sigma$, let $\sA^{\sigma}$ denote the smallest simplicial subset of $\Delta[r] \times \Delta[s]$ containing
the boundary $\partial (\Delta[r] \times \Delta[s])$ and all $(r+s)$-simplices corresponding to shuffles $\tau$
where $\tau<\sigma$.

\medskip
One of the main ingredients in our proof of proposition \ref{prop-addrel} is the following lemma. We postpone its proof
to appendix~\ref{sec-app2}. (That appendix also has a definition of \emph{inner generalized horn extension}.)

\begin{lem}\label{lem-Ainneranodyne}
The inclusion
$\sA^{\sigma} \hookrightarrow \sA^{\sigma} \cup \Delta[r+s]_{\sigma}$
is an inner generalized horn extension if $\sigma$ is not the maximal shuffle. If $\sigma$ is the maximal shuffle,
this inclusion is the homotopy cobase change of $\partial \sigma: \partial \Delta[r+s] \to \sA^\sigma$
along $\partial \Delta[r+s] \hookrightarrow \Delta[r+s]$.
\end{lem}

\subsection{Decompositions of $\cconfig(M|D)$ and the proof of proposition \ref{prop-addrel}}

The configuration category analogs of lemma \ref{lem-Ainneranodyne} are the following two lemmas.
A choice of diagram $\sD$ as in~\eqref{eqn-filled} and notation~\ref{not:diagrams} is understood.
This determines a map of simplicial sets $\Delta[r]\times\Delta[s]\to \FF$. Then a choice of $(r,s)$-shuffle
$\sigma$ determines $\sA^{\sigma}\subset \Delta[r]\times\Delta[s]$, and consequently
we can say that $\sA^\sigma$ is a simplicial set over $\FF$.

Let $A^\sigma\subset\FF$ be the image of $\sA^\sigma$  under the map $\Delta[r]\times\Delta[s]\to \FF$ determined by diagram $\sD$.
We may also write $\bar\sigma$ for the simplicial subset of $\FF$ generated by
the image of $\sigma$ in $\FF$.

\begin{lem}\label{lem-conAinneranodynereal}
For $r,s>0$ let $\sigma$ be an $(r,s)$-shuffle which is not the maximal one. Let $C$ be a
saturated simplicial subset of $\FF$ which contains $A^\sigma$.
Then the inclusion
\[
\cconfig(M|C) \to \cconfig(M|C\cup\bar\sigma)
\]
is a weak equivalence in the Segal space model structure. Moreover
$C\cup\bar\sigma$ is again saturated.
\end{lem}
\proof The idea is that this should follow from lemma~\ref{lem-Ainneranodyne}, but there is a difficulty
because our simplicial map from $\Delta[r]\times\Delta[s]$ to $\FF$ need not be injective.
We assume that $\bar \sigma$ is not contained in $C$, otherwise there is nothing to show. It suffices to show the following:
\begin{itemize}
\item[(i)] there is a (strict) pushout square of simplicial sets and maps
\begin{equation*}
\xymatrix@R=16pt{
\Lambda^S[r+s] \ar[d] \ar[r] & \Delta[r+s] \ar[d] \\
C \ar[r]^-{\textup{incl.}} & C\cup \bar\sigma
}
\end{equation*}
where the upper row is an inner generalized horn inclusion.
\end{itemize}
Given (i), it is then is a formal consequence (lemma~\ref{lem-coexcisive}) that we get another homotopy pushout square
\[
\xymatrix@R=16pt{
\Lambda^S[r+s]\otimes^\LL Z \ar[d] \ar[r] & \Delta[r+s]\otimes^\LL Z \ar[d] \\
\cconfig(M|C) \ar[r] & \cconfig(M|C\cup \bar\sigma)
}
\]
where $Z$, also known as $R(\sD)$, is the fiber over $\sigma$ of $\cconfig(M)_{r+s}\to \FF_{r+s}$.
Since the top row in this last diagram is a weak equivalence in the model structure of (complete) Segal spaces,
the same is true for the bottom row, finishing the proof.

To establish (i), we rely on the ``known'' pushout square
of lemma~\ref{lem-Ainneranodyne},
\[
\xymatrix@R=16pt@C=50pt{
\Lambda^S[r+s] \ar[d] \ar[r] & \Delta[r+s] \ar[d] \\
\sA^\sigma \ar[r] & \sA^\sigma\cup \Delta[r+s]_\sigma
}
\]
which is a union-intersection square of simplicial subsets of $\Delta[r]\times\Delta[s]$.
The generalized horn $\Lambda^S[r+s]$ here is the same as in the pushout square of (i).

We begin by collecting some properties
of faces of $\Delta[r+s]_\sigma$ which are not in
$\sA^\sigma\subset \Delta[r]\times\Delta[s]$. The faces of $\sigma$
correspond to nonempty subsets $T$ of $[r+s]$; we write $\sigma_T$ for the face corresponding to $T$.
We call $x\in [r+s]\smin [0]$ \emph{horizontal}, resp.~\emph{vertical}, if $\sigma(x)-\sigma(x-1)=(1,0)$,
resp.~$\sigma(x)-\sigma(x-1)=(0,1)$.
\begin{itemize}
\item[a)] $\sigma_T\notin \sA^\sigma$ implies $0\in T$ and $r+s\in T$.
If $x\in [r+s]\smin T$, then $x-1\in T$ and $x+1\in T$ and $\sigma(x+1)-\sigma(x-1)=(1,1)$.
\end{itemize}
Otherwise the face corresponding to $T$ is contained in $\partial(\Delta[s]\times\Delta[r])$, which is contained in $\sA^\sigma$
by definition.
\begin{itemize}
\item[b)] In the notation of a), the element $x$ is horizontal and $x+1$ is vertical.
\end{itemize}
Otherwise $\sigma_T$ is a face of some shuffle $\tau$ which is $<\sigma$.
(Suppose that $x$ is vertical for a contradiction and define $\tau(x)=\sigma(x)+(1,-1)$ for that $x$ only,
and $\tau(y)=\sigma(y)$ for all $y\ne x$.) --- This has the following consequence:
\begin{itemize}
\item[(ii)] If $\sigma_T\notin \sA^\sigma$, then it cannot land in $C\subset \FF$.
\end{itemize}
Suppose it does land in $C$; then, since $C$ is saturated, $\sigma$ itself must land in $C$, contradiction. --- It remains to
show:
\begin{itemize}
\item[(iii)] Two distinct $k$-dimensional faces of $\sigma$ which are not in $\sA^\sigma$
have distinct images in $\FF$ (under the map determined by $\sD$).
\end{itemize}
We need two more observations for the proof of (iii).
\begin{itemize}
\item[c)] $\sigma$ maps to a nondegenerate element of $\FF_{r+s}$.
\end{itemize}
Suppose for a contradiction that it maps to a degenerate element. Then for some $x\in [r+s]\smin[0]$
the morphism in $\sD$ with source position $\sigma(x)$ and target position $\sigma(x-1)$ is bijective.
If $x=1$, it follows that $\sigma$ is (in $\FF$) a degeneracy of $\sigma_{[r+s]\smin[0]}$. But
$\sigma_{[r+s]\smin[0]}$ violates a), so that it belongs to $\sA^\sigma$, which would mean that $\sigma$ lands in $C$;
contradiction. A similar argument shows that $x=r+s$ can be excluded. If $1<x<r+s$, then $\sigma$ is (in $\FF$)
a degeneracy of $\sigma_{[r+s]-\{x\}}$ and also of $\sigma_{[r+s]-\{x-1\}}$. But one of these two must
violate b), so that one of them belongs to $\sA^\sigma$ and consequently lands in $C$. Then again $\sigma$ lands in $C$,
contradiction.
\begin{itemize}
\item[d)] If $\sigma_T\notin \sA^\sigma$, and $\sigma_T$ is a proper face
of $\sigma$, then $\sigma_T$ maps to a nondegenerate element of $\FF_{|T|-1}$.
\end{itemize}
For suppose that it maps to a degenerate element. Then for some $x\in [r+s]\smin T$, as in a),
the morphism in diagram $\sD$ with source position $\sigma(x+1)$ and target position $\sigma(x-1)$
is bijective; here we use a) and c). In $\FF$, the image of $\sigma_T$ is a degeneracy of the image of $\sigma_{T\smin\{x+1\}}$.
But $\sigma_{T\smin\{x+1\}}$ violates condition a). Therefore it belongs to $\sA^\sigma$, and so $\sigma_T$
itself has to land in $C$. This contradicts (ii), which we have proved.

Now it is easy to prove (iii). The simplex $\sigma$ determines a sub-diagram of $\sD$ which is a
string of $r+s$ composable maps $f_x$ of finite sets:
\[  P_0 \xleftarrow{f_1} P_1 \xleftarrow{f_2} P_2 \cdots P_{r+s-1} \xleftarrow{f_{r+s}} P_{r+s}. \]
Each of the $f_x$ is either a surjection or an injection, but never a bijection.
The proper faces $\sigma_T$ of $\sigma$ correspond to strings of maps obtained by composing
some of the $f_x$. A proper face $\sigma_T$ is not in $\sA^\sigma$ if and only if, for every $x\notin T$,
the map $f_x$ is injective and $f_{x+1}$ is surjective. The string of maps corresponding to
$\sigma_T$ is obtained by replacing each of the substrings $(f_x,f_{x+1})$ by a single map,
the composition $f_xf_{x+1}$. That map $f_xf_{x+1}$ is neither surjective nor injective. Therefore
we can recover $T$, knowing only the image of $\sigma_T$ in $\FF_{|T|-1}$, by asking which
1-dimensional faces of that element in $\FF_{|T|-1}$ correspond to isomorphism classes
of maps which are surjective or injective. This proves (iii), and completes the proof
of lemma~\ref{lem-conAinneranodynereal} except perhaps for the claim that $C\cup \bar\sigma$
is again saturated. But that is trivial since $A^{\sigma}\subset C$ by assumption. \qed

\begin{lem}\label{lem-addspherereal}
For $r,s>0$ let $\sigma$ be the maximal $(r,s)$-shuffle. Let $C$ be a
saturated simplicial subset of $\FF$ which contains $A^\sigma$, but does not contain
the image of $\sigma$ in $\FF$. Then there is a degreewise
homotopy pushout square of simplicial spaces:
\[
\xymatrix@R=16pt{
\partial \Delta[r+s] \otimes^\LL R(\sD) \ar[r] \ar[d] & \ar[d] \Delta[r+s] \otimes^\LL R(\sD) \\
	\cconfig(M|C) \ar[r]  & \cconfig(M|C\cup\bar\sigma)
}
\]
\end{lem}
\proof  Our assumptions make it obvious that there is
a (strict) pushout square of simplicial sets and maps
\[
\xymatrix@R=16pt{
\partial\Delta[r+s] \ar[d] \ar[r] & \Delta[r+s] \ar[d] \\
C \ar[r]^-{\textup{incl.}} & C\cup \bar\sigma
}
\]
where the top left-hand term $\partial\Delta[r+s]$ is explained by lemma~\ref{lem-Ainneranodyne}.
It is a formal consequence that we get a homotopy
pushout square
\[
\xymatrix@R=16pt{
\partial\Delta[r+s]\otimes^\LL Z \ar[d] \ar[r] & \Delta[r+s]\otimes^\LL Z \ar[d] \\
\cconfig(M|C) \ar[r] & \cconfig(M|C\cup \bar\sigma)
}
\]
where $Z$, also known as $R(\sD)$, is the fiber over $\sigma$ of $\cconfig(M)_{r+s}\to \FF_{r+s}$. \qed

\proof[Proof of proposition \ref{prop-addrel}]
Choose an enumeration of the set of $(r,s)$-shuffles, say
$i\mapsto \sigma_i$ where $i\in J:=\{1,2,\dots,\binom{r+s}{s}\}$, such that $\sigma_i<\sigma_j$ only if $i<j$.
Let
\[  B^i:= A^{\sigma_1}\cup A^{\sigma_2}\cup \cdots \cup A^{\sigma_i} \subset \FF. \]
All the $B^i$ are saturated because all the $A^j$ are saturated.
By lemma~\ref{lem-conAinneranodynereal}, the inclusion
\[  \cconfig(M|C\cup B^i) \hookrightarrow \cconfig(M|C\cup B^{i+1})  \]
is a weak equivalence in the Segal space model structure if $i,i+1\in J$.
By lemma~\ref{lem-addspherereal}, for the maximal element
$i$ of $J$ the inclusion
\[  \cconfig(M|C\cup B^{i}) \lra \cconfig(M|C\cup D)  \]
fits into a degreewise homotopy pushout square of simplicial spaces
\[
\xymatrix@R=18pt{
\partial \Delta[r+s] \otimes^\LL R(\sD) \ar[r] \ar[d] & \Delta[r+s] \otimes^\LL R(\sD) \ar[d] \\
	\cconfig(M|C\cup B^i) \ar[r] & \cconfig(M|C\cup D)
}
\]
which uses $B^i\cup \bar\sigma_i=D$. (For the maximal $i\in J$ we have
$B^i=A^{\sigma_i}$.) \qed

\proof[Proof of theorem \ref{thm-hocoend}] Choose an enumeration, say $i\mapsto \sD_i$ where $i\in \NN$,
of the isomorphism types of diagrams $\sD$ like~\eqref{eqn-filled} such that $D=\im(\sD)\subset C$ and $\sD^\natural$ is nondegenerate. We impose
some conditions on the enumeration. First condition,
$D_i=\im(\sD_i)$ is not contained in the union of the $D_j$ where $j<i$; second condition, the ``boundary''
of $D_i$ is contained in the union of the $D_j$ where $j<i$. (The boundary of $D_i$ is the simplicial
subset of $\FF$ which is the image of the composition $\partial(\Delta[r]\times\Delta[s])\hookrightarrow
\Delta[r]\times\Delta[s]\to \FF$, second arrow determined by $D_i$, assuming this has size $r\times s$.)
Let
\[ C_i=\bigcup_{j\le i} D_j~, \qquad
C_{<i}=\bigcup_{j< i} D_j\,. \]
Then $\bigcup_i C_i=C\subset\FF$.
It is not hard to see that there is a homotopy pushout square (degreewise, in the category of simplicial spaces)
\[
\xymatrix{
\partial (\Delta[r] \times \Delta[s]) \otimes^\LL R(\sD_i) \ar[r] \ar[d] & \ar[d]
\Delta[r] \times \Delta[s] \otimes^\LL R(\sD_i)  \\
	\Delta[-] \times \Delta[-] \otimes^\LL_{\Delta \times \Delta} R^{C_{<i}} \ar[r] & \Delta[-] \times \Delta[-]
\otimes^\LL_{\Delta \times \Delta} R^{C_i}
}
\]

Comparison with proposition \ref{prop-addrel}
completes the proof, by induction on $i$ and passage to the colimit, $i\to \infty$. \qed

\section{Defects} \label{sec-def}
\begin{defn} The \emph{defect} of a diagram of finite sets and maps
\begin{equation} \label{eqn-defect}
\xymatrix{
S_0 & \ar[l]_-{f_1}  S_1 & \ar[l]_-{f_2}  S_2 & \ar[l]  \cdots & \ar[l]_-{f_{t-1}}  S_{t-1} & \ar[l]_-{f_t}  S_t
}
\end{equation}
is the non-negative integer
$\sum_{i=0}^t |S_i| - \sum_{i=1}^t |\im(f_i)|$.
We can think of  \emph{defect} as a degreewise function on the simplicial set
$\FF$.
\end{defn}

\begin{lem} The defect does not increase when face operators are applied in the simplicial set $\FF$.
It is unchanged when degeneracy operators are applied. \qed
\end{lem}

\begin{cor} \label{cor-defmax} The defect of a diagram such as~\emph{(\ref{eqn-defect})} is not less than the maximum
of the $|S_i|$, where $i=0,1,\dots,t$. \qed
\end{cor}

\begin{cor} \label{cor-count}
For a positive integer $\alpha$, there are only finitely many isomorphism classes of diagrams like
\emph{(\ref{eqn-defect})} such that the defect is $\le \alpha$ and none of
the arrows are bijective. \qed
\end{cor}

\begin{expl} \label{expl-defct} Suppose that in diagram~(\ref{eqn-defect}) the arrows $f_t,\dots,f_{t+1-r}$ are injective and
the arrows $f_{t-r},\dots,f_1$ are surjective. Then the defect is equal to $|S_{t-r}|$. Or suppose that in diagram~(\ref{eqn-defect}) the arrows $f_t,\dots,f_{t+1-r}$ are surjective and
the arrows $f_{t-r},\dots,f_1$ are injective. Then the defect is equal to $|S_0|+|S_t|$.
\end{expl}

\begin{lem} \label{lem-defdiag}
Let $C$ be the simplicial subset of $\FF$ defined by the condition
\[ \textup{defect }\le \alpha. \]
Then $C$ is accessible. Indeed $C=E^\alpha$; it is the union of all $D=\im(\sD)$
where $\sD$ is a diagram of the form~\eqref{eqn-filled} in which all sets $S_{i,j}$
have cardinality $\le \alpha$.
\end{lem}

\proof It is enough to show that every simplex in $\FF$ represented by a
string like~\eqref{eqn-defect} which has defect $\alpha$ (exactly) is contained
in $D=\im(\sD)$ for an appropriate diagram $\sD$ as in~\eqref{eqn-filled}, subject to the
condition $|S_{0,s}|\le \alpha$. The defect does not change if we decompose each $f_i$ into two arrows,
surjection followed by injection. Therefore we may assume without loss that $t$ is even and $f_i$
is injective for $i$ odd, surjective for $i$ even. This gives the following picture (with some relabeling, $S_j$ turns into
$S_{j,j}$):
\begin{equation} \label{eqn-staircase}
\begin{aligned}
\xymatrix@M=7pt@R=14pt@C=18pt{
 &  &  &  & S_{t,t} \\
  & & & S_{t-1,t-1} \ar@{<-<}[r] & S_{t,t-1} \ar@{<<-}[u] \\
 &  & \cdots  \ar@{<-<}[r] & S_{t-1,t-2} \ar@{<<-}[u]  &  \\
 & S_{1,1}  \ar@{<-<}[r]  & \cdots   &  &  \\
S_{0,0} \ar@{<-<}[r] & S_{1,0} \ar@{<<-}[u]  &   &  &
}
\end{aligned}
\end{equation}
Now we need to complete this to a diagram $\sD$ like~\eqref{eqn-filled}, taking $r=s=t$.
Such a $\sD$ is not unique, but it exists and it is also unique up to non-unique isomorphism (i.e., natural
isomorphism of diagrams of sets) if we impose the following interesting condition: for every $(t,t)$-shuffle
$\sigma\co [t]\to [t]\times[t]$, the string
of maps $\sigma^*(\sD)$ has the same defect $\alpha$. \qed

\begin{cor} \label{cor-defdiag} $\FF$ is accessible.  \qed
\end{cor}

\begin{prop} \label{prop-def} The inclusion of $\cconfig(M|E^\alpha)$ in the truncated
configuration category $\cconfig(M;\alpha)$ is a weak equivalence in the Segal space model structure.
\end{prop}

\proof We may view $\FF$ as a \emph{graded} set (and we use lemma~\ref{lem-defdiag}).
Let $T$ be the part of that graded set consisting of the elements which are represented by diagrams like~(\ref{eqn-defect})
where all $S_i$ have cardinality $\le \alpha$, none of the $f_i$ are bijective and the defect is $>\alpha$.
Our plan is to prove the proposition by showing that $\cconfig(M;\alpha)$ can be obtained from $\cconfig(M|E^\alpha)$ by a sequence
of (parametrized) inner horn extensions. (These are weak equivalences in the Segal space model structure.)
Due to the parametrization, each inner horn extension will add all the connected components of $\cconfig(M;\alpha)$
corresponding to two elements $y, z$ in $T\subset \FF$, both of the same defect and in
adjacent degrees, say $n-1$ and $n$, and so that $y$ is an inner face of $z$,
which means $y=d_j z$ for some $j\in \{1,2,\dots,n-1\}$.
Therefore we have to write $T$ as a disjoint union
\[ T = T_I \sqcup T_{II} \]
and we have to make a bijection $\psi\co T_{II}\to T_{I}$ such that the following conditions are satisfied:
\begin{itemize}
\item[(i)] For $z\in T_{II}$ of degree $n$, we can write $\psi(z)=d_j(z)$ for some $j$ such that $0<j<n$.
(This $j$ will turn out to be unique.)
\item[(ii)] The face operator $d_j$ induces a weak equivalence of spaces
from the part of $\cconfig(M;\alpha)$ projecting to $z$ to the part of $\cconfig(M;\alpha)$ projecting to $\psi(z)$.
\item[(iii)] The defect of $\psi(z)$ is equal to the defect of $z$.
\end{itemize}
(Other conditions will be spelled out later.) Now we describe $T_{I}$ and $T_{II}$. For motivation we recall
example~\ref{expl-defct}.
We can represent every element $z$ of $T$ by a diagram like~(\ref{eqn-defect}) where
\begin{itemize}
\item[] $f_t,f_{t-1},\dots,f_{t+1-r}$ are proper injections (for some $r\ge 0$) whereas $f_{t-r}$ is not injective;
\item[] $f_{t-r},\dots,f_{t+1-r-s}$ are proper surjections (for some $s\ge 0$) whereas $f_{t-r-s}$ is not surjective
\end{itemize}
and $r+s<t$. (If $r+s=t$, then $z$ has defect $\le \alpha$, therefore $z\notin T$, contradiction.)
We say that $z\in T_{II}$ if $f_{t-r-s}$ is (properly) injective, and $z\in T_{I}$ otherwise. This defines the decomposition
of $T$. Note that if $z\in T_{II}$, then $s\ge 1$, for otherwise $f_{t-r}$ must be both injective and non-injective.
We define $\psi\co T_{II} \lra T_{I}$ by $\psi(z)=d_{t-r-s}z$ where $d_{t-r-s}$ is the (inner) face operator which composes
the surjection $f_{t+1-r-s}$ with the injection $f_{t-r-s}$. Then
$\psi$ is bijective. The conditions (i),(ii) and (iii) just above
are satisfied.

Finally we need to put a partial ordering on $T_{II}$ such that the following holds:
\begin{itemize}
\item[(a)] if $z\in T_{II}$ and $d_iz$ is a face of $z$ other than the distinguished one, $d_{t-r-s}z$, then
we have $d_iz<z$ if $d_iz\in T_{II}$ and $\psi^{-1}(d_iz)<z$ if $d_iz\in T_I$.
\end{itemize}
In search of such a partial ordering, we make the following observations, continuing in the
notation of (a).
\begin{itemize}
\item[(b)] If $z\in T_{II}$ and $d_iz\in T_{II}$, then obviously the degree of $d_iz$ is less than the
degree of $z$.
\item[(c)] It can happen that $d_iz\notin T$. In that case $d_iz$ can still be a degeneracy of some
$y\in T$, but then the degree of $y$ is less than the degree of $z$.
\item[(e)] If neither (b) nor (c) applies, then $d_iz\in T_{I}$ and either $i=t-r-s-1$ so that $d_i$ is the face
operator which composes $f_{t-r-s}$ and $f_{t-r-s-1}$ in the description of $z$, or $i=t-r$ and $r>0$, so that $d_i$ is the face
operator which composes the properly surjective $f_{t+1-r}$ with the properly injective $f_{t-r}$.
In the first case $s$ increases as we pass
from $z$ to $\psi^{-1}(d_iz)$ while $r$ is unchanged. In the second case $f_{t-r}f_{t+1-r}$ is neither surjective
nor injective, so that $r$ decreases as we pass from $z$ to $\psi^{-1}(d_iz)$. In both cases the degree
of $z$ is equal to the degree of $\psi^{-1}(d_iz)$.
\end{itemize}
Now define $\varphi\co T_{II}\lra \NN\times\NN\times(\ZZ\smin\NN)$ by $z\mapsto (|z|,r,-s)$. Give $\NN$ and $\ZZ\smin\NN$
the standard total orderings. Use the (resulting) lexicographic ordering on the target of $\varphi$. We obtain a partial
ordering on $T_{II}$ by saying $z<z'$ if and only if $\varphi(z)<\varphi(z')$. This has the property (a). Furthermore,
the preimage under $\varphi$ of an element in $\NN\times\NN\times(\ZZ\smin\NN)$ is finite, since $\pi_0\cconfig(M;\alpha)$
is finite in each degree. \qed

\medskip
\proof[Proof of theorem~\ref{thm-finpres}] This is immediate from proposition~\ref{prop-def} and corollary~\ref{cor-count}. \qed

\begin{rem} The above proof can also be read as follows. Let $X\subset \cfin$ be the preimage
of $E^\alpha\subset\FF$ under the projection. Let $\cfin(\alpha)$ be the full (complete) Segal subspace of
$\cfin$ spanned by the objects of cardinality $\le \alpha$. (In each degree $n$, this is a union of connected
components of $\cfin_n$.) \emph{Then the inclusion
\[   X \hookrightarrow \cfin(\alpha) \]
is a presentation.} For $\alpha\ge 2$ it is not a finite presentation, but by corollary~\ref{cor-count}
it is always finite dimensional in the sense that
$X$ is Reedy cofibrant and $X=\skel_r X$ for some $r\ge 0$ depending on $\alpha$.
(Here $\skel_r$ is the functor from simplicial spaces to simplicial spaces which first restricts to $\Delta_{\le r}$
and then makes a left Kan extension along $\Delta_{\le r}\hookrightarrow \Delta$.) Finally we can re-construct
$\cconfig(M;E^\alpha)$ by means of a homotopy pullback square and strict pullback square
\[
\xymatrix@R=18pt@M=5pt{
\cconfig(M|E^\alpha) \ar[d] \ar[r] & \ar[d]  \cconfig(M;\alpha) \\
X \ar[r] &  \cfin(\alpha)
}
\]
Then it is obvious that $\cconfig(M|E^\alpha)$ is homotopically compact as simplicial space, and also that the upper horizontal arrow is a
presentation.
\end{rem}

\appendix

\section{Products of representable simplicial sets} \label{sec-app2}

In this appendix, we prove lemma \ref{lem-Ainneranodyne}. In preparation for that, we collect some
definitions and known results.

\begin{defn}
Let $S \subset [n]$ be a proper subset. The \emph{generalized horn}
\[ \Lambda^S[n] \subset \Delta[n] \]
is the union of the codimension one faces indexed by $i \in S$, i.e., the union of $(\delta_i)_*(\Delta[n-1])$ for $i\in S$,
where $\delta_i\co [n-1]\to [n]$ is the monotone injection which skips $i\in [n]$.
An \emph{inner generalized horn} is a generalized horn $\Lambda^S[n]$ for which $S$ is not an interval, i.e.,
there exist $s_1,t,s_2\in [n]$ such
that $s_1<t<s_2$ and $s_1,s_2\in S$ while $t \notin S$.
\end{defn}
The usual horns are the $\Lambda^S[n]$ where the complement of $S$ is a singleton.

\begin{lem}\label{lem-inneranodyne} \cite[Prop.2.12]{Joyal08}
Inner generalized horns are inner anodyne.\footnote{Joyal writes \emph{mid anodyne}. His notation for generalized inner horns
is in conflict with ours, but ours is in agreement with lecture notes by Rezk.} \qed
\end{lem}
\begin{proof}[Proof of lemma \ref{lem-Ainneranodyne}]
Since inner anodyne maps are closed under cobase change, it is enough to show that the inclusion
\[
\iota : \sA^{\sigma} \cap \Delta[r+s]_{\sigma} \to \Delta[r+s]_{\sigma}
\]
is inner anodyne when $\sigma\ne \ulcorner$.

Let us abbreviate $\sA^{\sigma} \cap \Delta[r+s]_{\sigma}$ as $\sB$. We will show (i) that $\sB$ is a union of $(r+s-1)$-dimensional
simplices and (ii) that there is an \emph{inner} face $d_{t}\sigma$ of $\sigma$ which is not in $\sB$. Condition (i) implies
that $\iota$ is a generalized horn. Condition (ii) implies that $\iota$ is inner:  the faces $d_0 \sigma$ and
$d_{r+s} \sigma$ belong to $B$ since they both belong to the boundary $\partial (\Delta[r] \times \Delta[s])$, and $0 < t < r+s$.
Therefore, once (i) and (ii) are established, lemma \ref{lem-inneranodyne} finishes the proof.

In order to prove that condition (i) holds, we must show that every simplex $\gamma$ of $\sB$ which is not a face of some other
simplex in $\sB$ is of dimension $r+s-1$. We can think of $\gamma$ as an injective order-preserving map,
\[  \gamma\co [m]\to [r+s]. \]
If $\gamma(0)\ne 0$ then $\gamma$ is a face of $d_0 \sigma$. But $d_0 \sigma$ is in $\sB$ since it is in the boundary
$\partial( \Delta[r] \times \Delta[s])$. As $\gamma$ is not a face of some other simplex in $\sB$, this implies
$\gamma = d_0 \sigma$ and so $\gamma$
is of dimension $r+s-1$. Hence, we can assume that $\gamma(0)=0$ and, arguing similarly, that $\gamma(m) = r+s$.

The restriction of $\sigma\co [r+s]\to [r]\times[s]$ to the interval $\{\gamma(k-1),...,\gamma(k)\}\subset [r+s]$
can be thought of as a $(p,q)$-shuffle with $(p,q)=\sigma(\gamma(k))-\sigma(\gamma(k-1))$. We denote this by $\sigma^k$.
Suppose that there is some $k$ such that the shuffle $\sigma^k$ is not minimal.
Then we have $\sigma(z)=(a,b)$, $\sigma(z+1)=(a,b+1)$ and $\sigma(z+2)=(a+1,b+1)$ in $[r]\times[s]$ for some $z$ such that
$\{z,z+1,z+2\}\subset \{\gamma(k-1),...,\gamma(k)\}$, and some $a\in [r]$ and $b\in [s]$.
Now $\gamma$ is a face of $d_{z+1}\sigma$.
But there is a unique $(r,s)$-shuffle $\tau$ with $\tau < \sigma$ and such that
$d_{z+1}\sigma = d_{z+1}\tau$. Then $\tau$ is an $(r+s)$-simplex in $\sA^\sigma$ and so $d_{z+1}\tau=d_{z+1}\sigma$
is an $(r+s-1)$-simplex of $\sA^\sigma\cap\Delta[r+s]_\sigma=\sB$.
Hence $\gamma$, being a face of $d_{z+1}\sigma$, is equal to $d_{z+1}\sigma$ and so has dimension $r+s-1$.

It remains to treat the case where each $\sigma^k$ is minimal. In that case any $(r,s)$-shuffle $\tau$ which
has $\gamma$ as a face is $\ge \sigma$, so that $\gamma$ is not a simplex in $\sA^\sigma$; contradiction.

We move on to condition (ii). Since $\sigma\ne \ulcorner$,
there exists $z\in [0,r+s-2]$ such that $\sigma(z)=(a,b)$, $\sigma(z+1)=(a+1,b)$ and $\sigma(z+2)=(a+1,b+1)$
for some $a\in [r-1]$ and $b\in[s-1]$. Then the face $d_{z+1} \sigma$ is not in $\sA^\sigma$, hence not in $\sB$.
\end{proof}


\begin{thebibliography}{10}
\bibitem{BoavidaWeissLong} P.~Boavida de Brito and M.S.~Weiss, \emph{Spaces of smooth embeddings and configuration categories}, J. Topology 11 (2018), 65--143.
\bibitem{Hirschhorn} P.~Hirschhorn, \emph{Model categories and their localizations}, Math. surveys and monographs 99,
Amer.Math.Soc., 2003.
\bibitem{Joyal08} A.~Joyal, \emph{The theory of quasi-categories and its applications}, Notes for course at CRM, Barcelona 2008.
http:/\!/mat.uab.cat/$\sim$kock/crm/hocat/advanced-course/Quadern45-2.pdf.
\bibitem{Rezk01} C.~Rezk, \emph{A model for the homotopy theory of homotopy theory}, Trans.Amer.Math.Soc. 353 (2001),
973--1007.
\end{thebibliography}
\end{document}